\documentclass[11pt,twoside, leqno]{article}

\usepackage{amssymb}
\usepackage{amsmath}
\usepackage{amsthm}
\usepackage{enumerate}

\allowdisplaybreaks

\def\tl{\widetilde}

\def\rr{{\mathbb R}}
\def\rn{{{\rr}^n}}
\def\rrm{{{\rr}^m}}
\def\rnm{{\rn\times\rrm}}

\def\zz{{\mathbb Z}}
\def\nn{{\mathbb N}}
\def\ch{{\mathcal H}}
\def\cc{{\mathbb C}}
\def\cd{{\mathcal D}}
\def\cs{{\mathcal S}}
\def\hr{{\mathcal R}}

\def\cq{{\mathcal Q}}

\def\cm{{\mathcal M}}

\def\ca{{\mathcal A}}
\def\ccc{{\mathcal C}}

\def\fz{\infty}
\def\az{\alpha}

\def\supp{{\mathop\mathrm{\,supp\,}}}
\def\esssup{{\mathop\mathrm{\,ess\ sup\,}}}

\def\lz{\lambda}

\def\ez{\epsilon}
\def\bz{\beta}

\def\gz{{\gamma}}

\def\Oz{{\Omega}}
\def\wz{\widetilde}

\def\boz{{\Omega}}
\def\vz{\varphi}
\def\tz{\theta}
\def\sz{\sigma}

\def\wz{\widetilde}

\def\hs{\hspace{0.2cm}}

\def\ls{\lesssim}

\def\lp{{L^p_w(\rnm)}}

\def\wfz{{\ca_\fz(\rnm;\,\vec A)}}
\def\wwp{{\ca_p(\rnm;\,\vec A)}}
\def\wwq{{\ca_q(\rnm;\,\vec A)}}

\def\hp{{H^p_w(\rnm;\,\vec{A})}}

\def\hat{\widehat}

\def\gfz{\genfrac{}{}{0pt}{}}

\def\dfrac{\displaystyle\frac}

\def\r{\right}
\def\lf{\left}

\newtheorem{thm}{Theorem}[section]
\newtheorem{lem}{Lemma}[section]
\newtheorem{prop}{Proposition}[section]
\newtheorem{rem}{Remark}[section]

\newtheorem{defn}{Definition}[section]

\numberwithin{equation}{section}

\begin{document}

\arraycolsep=1pt

\title{{\vspace{-5cm}\small\hfill\bf Sci. China Math., to appear}\\
\vspace{4.5cm}\Large\bf
Anisotropic Singular Integrals in Product Spaces
\footnotetext{\hspace{-0.35cm} {\it
2000 Mathematics Subject Classification}. { Primary 42B30; Secondary
42B20, 42B25, 42B35.}
\endgraf{\it Key words and phrases.} expansive dilation,
Muckenhoupt weight, product space, Hardy space, bump function,
singular integral.\endgraf
This work was partially supported by
Start-up Funding Doctor of Xinjiang University (Grant No. BS090109),
National Science Foundation of US (Grant No. DMS 0653881)
and National Natural Science Foundation of China (Grant No. 10871025).
\endgraf $^\ast$ Corresponding author.
}}
\author{Baode Li, Marcin Bownik,
Dachun Yang$^\ast$ and Yuan Zhou}
\date{ }
\maketitle

\begin{center}
\begin{minipage}{13.5cm}\small
{\noindent{\bf Abstract.} In this paper, the authors introduce a class
of product anisotropic singular integral operators,
whose kernels are adapted to the action of a pair
$\vec A\equiv(A_1,\,A_2)$ of expansive dilations on
$\mathbb R^n$ and $\mathbb R^m$, respectively. This class
is a generalization of product singular integrals with
convolution kernels introduced in the isotropic
setting by Fefferman and Stein [Adv. in Math. 45 (1982), 117--143].
The authors establish the boundedness of these operators
in weighted Lebesgue and Hardy spaces with weights in
product $A_\infty$ Muckenhoupt weights on $\mathbb R^n \times \mathbb R^m$.
These results are new even in unweighted setting for product
anisotropic Hardy spaces.}
\end{minipage}
\end{center}

\section{Introduction\label{s1}}

\hskip\parindent The theory of Hardy spaces and singular integrals
plays an important role in harmonic analysis and partial
differential equations; see, for example, \cite{fs72,g1,g2,s93}. There were several
directions of extending Hardy and other function space
theory from Euclidean spaces to other domains and non-isotropic
settings; see, for example, \cite{b1,ct75,ct77,cw77,fs82,st87,
st89,ht05,tr97,v06,cs06}. A significant effort was devoted in
developing a theory of Hardy spaces and
singular integrals on product domains. This direction was initiated
by Gundy and Stein \cite{gs79} with R. Fefferman, Nagel and Stein
among its main contributors  \cite{f87, fs82-1, fs82, ns04}. In particular,
Fefferman and Stein \cite{fs82-1}
introduced a class of product singular integrals with convolution
kernels and established their boundedness in Lebesgue spaces.
Fefferman further proved the boundedness of certain singular
integrals from product Hardy spaces to Lebesgue spaces in \cite{f87}
and also established some weighted boundedness in \cite{f88}.

The goal of this paper is to extend some of the existing
isotropic product Hardy space theory to the non-isotropic
setting associated with expansive dilations. Let $A_1$
and $A_2$ be expansive dilations, respectively, on
${\mathbb R}^n$ and ${\mathbb R}^m$. Let $w$ be a
product $A_\infty$ Muckenhoupt weight associated with a
pair of dilations, $\vec A\equiv(A_1,\,A_2)$. Recently,
the authors \cite{blyz1} developed the theory of
weighted anisotropic product Hardy spaces $H^p_w(\mathbb
R^n\times\mathbb R^m;\,\vec A)$ with $p\in(0,\,1]$. Motivated by Bownik \cite{b1}
and Nagel-Stein \cite{ns04}, in this paper,
we introduce a class of anisotropic singular integrals
on $\mathbb R^n\times\mathbb R^m$,
whose kernels are adapted to $\vec A$ in the sense of Bownik and
have vanishing moments defined via bump functions
in the sense of Stein. Then, we establish
the boundedness of these anisotropic singular integrals
on weighted Lebesgue spaces
$L^q_w(\mathbb R^n\times\mathbb R^m)$ with $q\in(1,\,\infty)$
and weighted Hardy spaces $H^p_w(\mathbb
R^n\times\mathbb R^m;\,\vec A)$ with $p\in(0,\,1]$.
These results are new even in the unweighted setting $w=1$.

We point out that the vanishing moments of singular integrals
defined via bump functions were originally introduced by Stein \cite{s93}.
To obtain the estimates for solutions of the Kohn-Laplacian
on certain classes of model domains in $\cc^N$, Nagel and
Stein \cite{ns04, ns06} introduced a class of singular integrals
including their product versions, whose vanishing moments are
defined via bump functions. Such a theory of product
singular integrals is also used in the analysis
on Heisenberg-type groups; see \cite{mrs95}.

To state our main results, we carefully define the class
of product anisotropic singular integral operators adapted
to the action of a pair $\vec A$ of expansive dilations.

\begin{defn}\rm\label{d1.1}
A real $n\times n$ matrix $A$ is an {\it expansive dilation},
shortly a {\it dilation}, if all its eigenvalues $\lz$
satisfy $|\lz|>1$.
{\it Throughout the whole paper, for the convenience,
we sometimes use $\rr^{n_1}$ and $\rr^{n_2}$ to denote, respectively,
$\rr^n$ and $\rr^m$.} For expansive dilation $A_i$ on $\rr^{n_i}$,
$i=1,\,2$, we always let $b_i\equiv |\det(A_i)|$ and $\vec A\equiv(A_1, A_2)$.
We also let $B^{(i)}_k$, $k\in\zz$,
be dilated balls and $\rho_i$ the step homogeneous-norm
associated with $A_i$ as in Definition \ref{d2.1}.
\end{defn}

\begin{defn}\label{d1.2}\rm
Let $N\in\nn$. A function $\psi$ on $\rn$ is called an {\it
$N$-normalized bump function associated to the ball $B_0$},
if $\supp\psi\subset
B_0$, and $\|\partial^\az\psi\|_{L^\infty(\rn)}\le1$ for all
$\az\in\zz_+^n$ with $|\az|\le N$.
A function $\psi$ on $\rn$ is called an {\it
$N$-normalized bump function associated to the ball $B_k$ with
$k\in\zz$} if and only if $\psi(A^k\cdot)$ is
an $N$-normalized bump function associated to the ball $B_0$.
\end{defn}

Let $\cd(\rnm)$ be the space
of all infinite differentiable functions with compact supports
endowed with the inductive limit topology and $\cd'(\rnm)$
its topological dual space. Also, let
$\boz_{n\times m}\equiv(\rnm)\setminus\{(x_1,x_2):\ x_1=0\ {\rm or}\
x_2=0\}$, $\nn\equiv\{1,\,2,\,\cdots\}$ and $\zz_+\equiv\nn\cup\{0\}$.

\begin{defn}\label{d1.3}\rm
Let $s_1,\, s_2\in\zz_+$. Let $T:\cd(\rnm) \to\cd'(\rnm)$ be
a continuous linear mapping. Then,
$T$ is called a {\it product anisotropic
singular integral operator} (PASIO) of order $(s_1,s_2)$,
if the following conditions are met:

(K0) $T$ has a distribution kernel $K$,
which is a continuous function on $\boz_{n\times m}$, such
that for all $\vz=\vz^{(1)}\otimes\vz^{(2)}\in\cd(\rnm)$
and $x_1 \not\in \supp \vz^{(1)}$, $x_2 \not\in \supp \vz^{(2)}$,
\[
T(\vz)(x_1,x_2)=\int_\rnm K(x_1-y_1,\,x_2-y_2)\vz^{(1)}(y_1)
\vz^{(2)}(y_2)\,dy_1\,dy_2;
\]

(K1) there exists a positive constant $C_1$ such that
for all $(x_1,\,x_2)\in\boz_{n\times m}$ with $\rho_i(x_i)=b_i^{\ell_i}$,
and for all $\az_i\in\zz_+^{n_i}$ with
$|\az_i|\le s_i$, $i=1,\,2$,
$$|\partial^{\az_1}_1\partial^{\az_2}_2[K(A_1^{\ell_1}\cdot,\,
A_2^{\ell_2}\cdot)](A_1^{-\ell_1}x_1,\, A_2^{-\ell_2}x_2)|\le C_1
[\rho_1(x_1)]^{-1}[\rho_2(x_2)]^{-1};$$

(K2) there exist $N_1,\,N_2\in\nn$ such that
for each $N_1$-normalized bump function $\psi^{(1)}$ associated
to $B^{(1)}_0$ and $N_2$-normalized bump function $\psi^{(2)}$
associated to $B_0^{(2)}$, and all $k_1,\,k_2\in\zz$,
$$\lf|\lf\langle
K, \, \psi^{(1)}(A_1^{k_1}\cdot)\otimes
\psi^{(2)}(A_2^{k_2}\cdot)\r\rangle\r|\le C_1;$$

(K3) for each $N_2$-normalized bump function $\psi^{(2)}$ associated
to $B_0^{(2)}$ and $k_2\in\zz$, there exists a continuous linear
operator $T^{\psi^{(2)},\,k_2}:\cd(\rn) \to \cd'(\rn)$ with a
distribution kernel $K^{\psi^{(2)},\,k_2}$, which is a continuous
function on $\rn\setminus\{0\}$, such
that for all $\vz^{(1)}\in\cd(\rn)$ and $x_1 \not\in\supp \vz^{(1)}$,
\[
T^{\psi^{(2)},\,k_2}(\vz^{(1)})=
T(\vz^{(1)}\otimes[\psi^{(2)}(A_2^{k_2}\cdot)])
=
\int_{\rn} K^{\psi^{(2)},\,k_2}(x_1-y_1) \vz^{(1)}(y_1) dy_1.
\]
Furthermore, for all $x_1 \neq 0$ with $\rho_1(x_1)=b_1^{\ell_1}$
and for all $\az_1\in\zz_+^n$ with $|\az_1|=s_1$,
\begin{eqnarray*}
&&|\partial_1^{\az_1}
[K^{\psi^{(2)},\,k_2}(A_1^{\ell_1}\cdot)](A_1^{-\ell_1}x_1)| \le
C_1 [\rho_1(x_1)]^{-1}.
\end{eqnarray*}
(K3) also holds with the roles of $x_1$ and $x_2$ interchanged.
\end{defn}

In the case when less regularity is desired, one can weaken
conditions (K1) and (K3) on the derivatives to more familiar
conditions on differences as in the work of Han and Yang
\cite{hy05} (see also \cite{hy09}).

\begin{defn}\label{d1.3A}\rm
In what follows, let $\sz_i$ for $i=1,\,2$ be as
in \eqref{e2.1} associated with $A_i$.
We say that $T$ is a {\it product anisotropic
singular integral operator} of order $0$,
if it satisfies Definition \ref{d1.3} with $s_1=s_2=0$.
Moreover, there exist $\ez_1,\ez_2>0$ such that for all
$(x_1,\,x_2)\in\boz_{n\times m}$ with  $\rho_i(x_i)=b_i^{\ell_i}$
and $h_i\in\rr^{n_i}$ with $\rho_i(h_i)\le b_i^{-2\sz_i}\rho_i(x_i)$,
we have
\begin{align*}
|\Delta_{h_1}^{(1)} K(x_1,x_2)| &\le C_1\frac{[\rho_1(h_1)]^{\ez_1}}
{[\rho_1(x_1)]^{1+\ez_1}}\frac 1{\rho_2(x_2)},
\\
|\Delta_{h_1}^{(1)}\Delta^{(2)}_{h_2} K(x_1,x_2)|&\le C_1
\frac{[\rho_1(h_1)]^{\ez_1}}{[\rho_1(x_1)]^{1+\ez_1}}
\frac{[\rho_2(h_2)]^{\ez_2}}{[\rho_2(x_2)]^{1+\ez_2}},
\\
|\Delta^{(1)}_{h_1} K^{\psi^{(2)},k_2}(x_1)|
&\le
C_1\frac{[\rho_1(h_1)]^{\ez_1}}{[\rho_1(x_1)]^{1+\ez_1}}.
\end{align*}
Here, we used difference operators
$\Delta_{h_1}^{(1)}K(x_1,\,x_2)\equiv K(x_1+h_1,x_2)-K(x_1,\,x_2)$
and $ \Delta_{h_2}^{(2)}K(x_1,\,x_2)\equiv
K(x_1,\,x_2+h_2)-K(x_1,\,x_2).$
The above estimates must also hold with the roles of $x_1$ and $x_2$ interchanged.
\end{defn}

Finally, we are ready to formulate the two main
results of this paper. Theorem \ref{t1.1} is a generalization
of a result of Fefferman and Stein \cite{fs82-1} from the
classical isotropic setting to the non-isotropic setting. Likewise,
Theorem \ref{t1.2} is a generalization of a
result of Han and Yang \cite{hy05} to the setting of weighted
anisotropic product Hardy spaces.

\begin{thm}\label{t1.1}
Let
 $w\in\ca_p(\rnm;\,\vec A)$ with $p\in(1,\,\fz)$. Then, a
 PASIO $T$ of order $0$ uniquely extends to a bounded operator
on $L^p_w(\rnm)$.
\end{thm}

\begin{thm}\label{t1.2} Let $w\in\wfz$ and $q_w$ be its
critical index as in \eqref{e2.4}. Let $s_1, s_2\in\zz_+$ and $p\in(0,\,1]$. If
\begin{equation}\label{e1.1}
s_i>(q_w/p-1)\log_{|\lz_{i,1}|}b_i \qquad\text{for }i=1,2,
\end{equation}
where $\lz_{i,1}$ is the smallest eigenvalue of $A_i$ in absolute value, then
a PASIO $T$ of order $(s_1+1,s_2+1)$ uniquely extends to a bounded
operator on $\hp$. Moreover, $T$ admits another unique
bounded extension to an operator $\hp \to L^p_w(\rnm)$.
\end{thm}

\begin{rem}\label{r1.2}\rm
Consider the classical case corresponding to the choice
of dyadic dilations $A_1=2I_{n_1}$, $A_2=2I_{n_2}$ and weight $w=1$.
Then, $q_w=1$, $\rho_i(x)=|x|^{n_i}$, and $\log_{|\lz_{i,1}|}b_i=n_i$ for $i=1,\,2$.
In this case, if $p\in(1,\,\fz)$ and
$\ez_i\in (0,\, 1/n_i]$, the boundedness on $L^p(\rnm)$
of product singular integrals as in Definition \ref{d1.3A}
follows from results of Nagel and Stein \cite{ns04}. On the other hand, if
$\max\{n_1/(n_1+\ez_1),\,n_2/(n_2+\ez_2)\}<p\le1$, then
the boundedness in $H^p(\rnm)$
of such product singular integrals was established by Han
and Yang \cite[Theorem 2]{hy05}.
\end{rem}

This paper is organized as follows. In Section \ref{s2}, we recall
some notation and known notions. The proofs of Theorems \ref{t1.1} and
\ref{t1.2} are presented in Sections \ref{s3} and \ref{s4}, respectively.
The methods used in these proofs borrow some ideas from
\cite{hy05} and \cite{ns04}; see also \cite{hy09} and \cite{hlll}.
However, unlike
\cite{hy05}, \cite{hy09} and \cite{hlll}, the discrete Calder\'on
reproducing formula with kernel
having compact support and the $g$-function characterization of the
product anisotropic Hardy spaces are not available. Instead, we use
the Lusin-area characterization with the kernels having no compact
support. To overcome these
additional difficulties, we invoke a decomposition technique of
kernels used by Nagel and Stein, see \cite[Lemma 3.5.1]{ns04} and
Lemma \ref{l3.1} below. Moreover, to prove Theorem \ref{t1.2}, we
use a variant of a key boundedness criterion established
in \cite[Corollary 6.1]{blyz1}, which reduces the boundedness of the
considered singular integrals to their behaviors on rectangular
atoms; see Lemma \ref{l3.4} below and also \cite[Corollary 1.1]{cyz}
for the corresponding result on $H^p(\rnm)$.

We finally make some conventions. Throughout this paper, we always
use $C$ to denote a positive constant that is independent of the main
parameters involved but whose value may differ from line to line.
Constants with subscripts do not change through the whole paper. We
use the symbol $f\ls g$ to denote $f\le Cg$,
and if $f\ls g\ls f$, we then write $f\sim g$. For all $x\in\rr$, we
denote $\lfloor x\rfloor$ by the {\it maximal integer no less than $x$}.

\section{Preliminaries}\label{s2}

\hskip\parindent In this section, we recall
basic facts about product Hardy spaces associated with expansive dilations.

By \cite[Lemma 2.2]{b1}, for a given expansive dilation $A$,
there exist an open ellipsoid $\Delta$ and $r\in (1,\,\fz)$
such that $\Delta\subset
r\Delta\subset A\Delta$. Moreover,
$|\Delta|=1$, where $|\Delta|$ denotes the $n$-dimensional Lebesgue
measure of the set $\Delta$.
Throughout the whole paper, we set
\begin{equation}\label{e2.1}
B_k\equiv A^k\Delta\ \mathrm{for}\ k\in \zz\
\mathrm{and\ let}\ \sz \ \mathrm{be\ the}\
minimum\ integer\ \mathrm{such\ that}\ 2B_0\subset A^\sz B_0.
\end{equation}
 Then $B_k$ is
open, $B_k\subset rB_k\subset B_{k+1}$ and $|B_k|=b^k$. Obviously,
$\sz\ge 1$. For any subset $E$ of $\rn$, let $E^\complement
\equiv\rn\setminus E$. Then it is easy to prove (see
\cite[p.\,8]{b1}) that for all $k,\,\ell\in\zz$,
\begin{eqnarray}
&&B_k+B_\ell\subset B_{\max\{k,\,\ell\}+\sz},\label{e2.2}\\
&&B_k+(B_{k+\sz})^\complement\subset (B_k)^\complement,\label{e2.3}
\end{eqnarray}
where $E+F$ denotes the algebraic sums $\{x+y:\, x\in E,\,y\in F\}$
of  sets $E,\, F\subset \rn$.

Recall that the homogeneous quasi-norm associated with $A$ was
introduced in \cite[Definition 2.3]{b1} as follows. For a fixed
dilation $A$, we always let $b\equiv|\det A|$.

\begin{defn}\label{d2.1}\rm
A \textit{homogeneous quasi-norm} associated with an expansive
dilation $A$ is a measurable mapping $\rho:\rn\to[0, \fz)$ such that

(i) $\rho(x)=0$ if and only if $x=0$;

(ii) $\rho(Ax)=b\rho(x)$ for all $x\in\rn$;

(iii) $\rho(x+y)\le H[\rho(x)+\rho(y)]$ for all $x,\, y\in\rn$,
where $H$ is a constant no less than $1$.

Define the \textit{step homogeneous quasi-norm} $\rho$ associated
with $A$ and $\Delta$ by setting, for all $x\in\rn$, $\rho(x)=b^k$
if $x\in B_{k+1}\setminus B_k$ or else $0$ if $x=0$.
\end{defn}

It was proved that all homogeneous quasi-norms associated with a
given dilation $A$ are equivalent (see \cite[Lemma 2.4]{b1}).
Therefore, for a given expansive dilation $A$, in what follows, for
convenience, we always use the step homogeneous quasi-norm $\rho$.
Moreover, from \eqref{e2.2} and \eqref{e2.3}, it follows that for
all $x,\,y\in\rn$,
$$\rho(x+y)\le b^\sz\max\lf\{\rho(x),\,\rho(y)\r\}\le
b^\sz[\rho(x)+\rho(y)].$$

The class of Muckenhoupt weights associated with $A$ was introduced
in \cite{bh}. For more details about weights, see
\cite{blyz,gr,g1,g2,st89}.

\begin{defn}\label{d2.3}\rm
Let $p\in[1,\,\fz)$, $A$ be a dilation and $w$ a nonnegative
measurable function on $\rn$. The function $w$ is said to belong to
the {\it weight class of Muckenhoupt} $\ca_p(\rn;\, A)$, if there
exists a positive constant $C$ such that when $p>1$,
$$\sup_{x\in\rn,\,k\in\zz}\lf\{\frac{1}{|B_k|}\int_{x+B_k}w(y)\,dy\r\}
\lf\{ \frac{1}{|B_k|}\int_{x+B_k}[w(y)]^{-1/(p-1)}\,dy\r\}^{p-1}\le
C,$$
$$\sup_{x\in\rn,\,k\in\zz}
\lf\{\frac{1}{|B_k|}\int_{x+B_k} w(y)\,dy\r\}\lf\{\mathop\esssup_{y\in x+B_k}
[w(y)]^{-1}\r\}\le C \qquad\text{when } p=1.$$
Moreover, the minimal constant $C$ as above is
denoted by $C_{A}(w)$.

Define $\ca_\fz(\rn;\,A)\equiv\cup_{1\le p<\fz} \ca_p(\rn;\,A)$.
\end{defn}

Product Muckenhoupt weights were first studied by R. Fefferman
\cite{f87}; see also \cite{s89}. Among several equivalent ways
of introducing product weights \cite[Theorem VI.6.2]{gr},
we adopt the following definition as in \cite{blyz1}.

\begin{defn}\label{d2.4}\rm
Let $\vec A=(A_1,A_2)$ be a pair of expansive dilations, respectively, on
${\mathbb R}^n$ and ${\mathbb R}^m$.
Let $p\in(1,\, \fz)$ and $w$ be a nonnegative
measurable function on $\rnm$. The function $w$
is said to be in the weight class of Muckenhoupt
$\ca_p(\rnm,\,\vec A)$, if $w(x_1,\, \cdot)\in \ca_p(\rrm;\,A_2)$
for almost every $x_1\in\rn$ and $\esssup_{x_1\in\rn}C_{A_2}
(w(x_1,\,\cdot))<\fz$, and $w(\cdot,\, x_2)\in
\ca_p(\rn;\,A_1)$ for almost every $x_2\in\rrm$ and
$\esssup_{x_2\in\rrm} C_{A_1}(w(\cdot,\, x_2))<\fz.$ In
what follows, let
$$C_{\vec A}(w)
\equiv\max\lf\{\mathop\esssup_{x_1\in\rn}C_{A_2}(w(x_1,\,\cdot)),
\ \mathop\esssup_{x_2\in\rrm} C_{A_1}(w(\cdot,\, x_2))\r\}.$$

Define $\wfz\equiv\cup_{1< p<\fz} \wwp$.
\end{defn}

For any $w\in\wfz$, define the critical index of $w$ by
\begin{eqnarray}\label{e2.4}
q_w\equiv\inf \{q\in(1,\,\fz):\, w\in\wwq\}.
\end{eqnarray}

Let $\cs(\rn)$ be the {\it space of Schwartz functions on $\rn$}.
For $\az\in\zz_+^n$ and $m\in\zz_+$, define seminorms
$\|\vz\|_{\az,\,m}\equiv\sup_{x\in\rn}[\rho(x)]^m
|\partial^\az\vz(x)|<\fz.$ It is
well-known that $\cs(\rn)$ forms a locally convex complete metric
space endowed with the seminorms
$\{\|\cdot\|_{\az,\,m}\}_{\az\in\zz_+^n,\,m\in\zz_+}$. The space
$\cs(\rn)$ coincides with the classical space of Schwartz functions;
see \cite[p.\,11]{b1}. The dual space of $\cs(\rn)$, namely, the space
of tempered distributions on $\rn$ is denoted by $\cs'(\rn)$.
Moreover, let
$\cs_0(\rn)\equiv \lf\{\psi\in\cs(\rn):\, \int_\rn \psi(x)\,dx=0 \r\}.$

For functions $\vz$ on $\rn$, $\psi$ on $\rnm$, $k, \ k_1,\
k_2\in\zz$, let
$\vz_k(x)\equiv b^{-k}\vz(A^{-k}x)$ and
$\psi_{k_1,\,k_2}(x)\equiv b_1^{-k_1}b_2^{-k_2}\psi(A_1^{-k_1}x_1,\,
A_2^{-k_2}x_2).$

Next, we introduce the product Lusin-area function and product
Littlewood-Paley $\vec g$-function following \cite{blyz1}.

\begin{defn}\label{d2.5}\rm
Let
$\vz^{(1)}\in\cs_0(\rr^{n})$ and $\vz^{(2)}\in\cs_0(\rr^{m})$.
Let $\vz \equiv \vz^{(1)} \otimes \vz^{(2)}$, where $\vz(x)
=\vz^{(1)}(x_1)\vz^{(2)}(x_2)$ for $x=(x_1,\,x_2)\in\rnm$.
For all $f\in\cs'(\rnm)$ and $x\in\rnm$,
define the {\it anisotropic product Lusin-area function} of $f$ by
$$ \vec S_\vz(f)(x)\equiv\lf\{\sum_{k_1,\, k_2\in\zz}
b^{-k_1}_1b^{-k_2}_2\int_{B^{(1)}_{k_1}\times
B^{(2)}_{k_2}}|\vz_{k_1,\, k_2}\ast f(x-y)|^2\,dy\r\}^{1/2}.$$
Define the {\it anisotropic product Littlewood-Paley $\vec
g$-function} of $f$ by
$$\vec g_\vz(f)(x)\equiv \lf\{\sum_{k_1,\,k_2\in\zz}|\vz_{k_1,\,k_2}\ast
f(x)|^2 \r\}^{1/2}.$$
\end{defn}

A distribution $f\in \cs'(\rnm)$ is said to {\it vanish weakly at
infinity} if for any $\vz^{(1)}\in\cs(\rrm)$ and
$\vz^{(2)}\in\cs(\rn)$, $f\ast\vz_{k_1,\,k_2}\to 0$ in $\cs'(\rnm)$
as $k_1,\,k_2\to \fz$. We denote by $\cs_\fz'(\rnm)$ the {\it set of all
$f\in \cs'(\rnm)$ vanishing weakly at infinity}.

We shall need the following existence result for functions
appearing in the Calder\'on formula, see \cite[Propositions 2.14 and 2.16]{blyz1}.

\begin{prop}\label{p2.1}
For $i=1,\,2$, let $s_i\in\zz_+$, $A_i$ be a dilation on $\rr^{n_i}$,
and $A_i^\ast$ its transpose. Then, there exist $\tz^{(i)},\, \psi^{(i)}\in
\cs(\rr^{n_i})$ such that:

(i) $\supp \tz^{(i)}\subset B_0^{(i)},\, \int_{\rr^{n_i}}
x_i^{\gz_i}\tz^{(i)}(x_i)\,dx_i=0$ for all $\gz_i\in(\zz_+)^{n_i}$
with $|\gz_i|\le s_i$, $\hat{\tz^{(i)}}(\xi_i)\ge C>0$ for $\xi_i$
in certain annulus,

(ii) $\supp\widehat{\psi^{(i)}}$ is compact and bounded away from
the origin,

(iii) $\sum_{j\in\zz}\widehat{\psi^{(i)}}((A_i^\ast)^j\xi_i)
\widehat{\tz^{(i)}} ((A^\ast_i)^j\xi_i)=1$ for all
$\xi_i\in\rr^{n_i}\setminus\{0\}$,

(iv) $\psi^{(i)}=\phi^{(i)}*\phi^{(i)}$ for some $\phi^{(i)} \in \cs(\rr^{n_i})$.
\end{prop}

Parts (i)--(iii) of Proposition \ref{p2.1} were proved in
the course of the proof  of \cite[Theorem 5.8]{bh}.
Part (iv) can be shown by a minor refinement of this argument
leading to the existence of $\phi^{(i)} \in \cs(\rr^{n_i})$
such that $(\widehat{\phi^{(i)}})^2=\widehat{\psi^{(i)}}$.

The following result says that the space $L^p_w(\rnm)$ can be
characterized by the Lusin-area $\vec S$-function and the
Littlewood-Paley $\vec g$-function. Proposition \ref{p2.2} is just
\cite[Theorem 3.2]{blyz1}, which also holds for
$\vec g$-function by a similar proof.

\begin{prop}\label{p2.2}
Let $\psi\equiv\psi^{(1)}\otimes \psi^{(2)}$ be as in Proposition \ref{p2.1}.
Then, the following are equivalent for $p\in(1,\,\fz)$:

(i)
$f\in L^p_w(\rnm)$,

(ii)
 $f\in\cs'_\fz(\rnm)$ and $\vec
S_\psi (f)\in L^p_w(\rnm)$,

(iii)  $f\in\cs_\fz'(\rnm)$ and
$\vec g_\psi(f)\in L^p_w(\rnm)$.

Moreover, for all $f\in
L^p_w(\rnm)$,
\[
\|f\|_{L^p_w(\rnm)}\sim \|\vec S_\psi(f)\|_{L^p_w(\rnm)}
\sim \|\vec g_\psi(f)\|_{L^p_w(\rnm)}.
\]
\end{prop}

Finally, we recall the definition of weighted
anisotropic product Hardy spaces in \cite{blyz1}.

\begin{defn}\label{d2.6}\rm
Let $w\in \wfz$ and $p\in (0,\ 1]$.
Let $\psi=\psi^{(1)} \otimes \psi^{(2)}$ be as
in Proposition \ref{p2.1}. The {\it weighted
anisotropic product Hardy space} is defined by
\begin{eqnarray*}
&\hp\equiv & \{f\in \cs'_\fz(\rnm):\,\|f\|_\hp\equiv\|\vec
S_\psi(f)\|_\lp<\fz\}.
\end{eqnarray*}
\end{defn}

\section{Proof of Theorem \ref{t1.1}}\label{s3}

\hskip\parindent To prove Theorem \ref{t1.1}, we need the following
decomposition technique of kernels, which adapts the
methods established by Nagel and Stein \cite[Lemma 3.5.1]{ns04} to our
setting. For the convenience of the reader, we present a detailed proof.

\begin{lem}\label{l3.1}
Let $N\in\nn$ and $\psi\in\cs_0(\rn)$. For any
$M>0$, there exists a constant $c>0$ and a decomposition $\psi=\sum_{k=0}^\fz
b^{-kM}\psi^{(k)}$, such that each $c\psi^{(k)}\in\cs_0(\rn)$
is an $N$-normalized bump function
associated to $B_k$.
\end{lem}

\begin{proof}
Let $\tz\in \ccc^\fz(\rn)$ be a non-negative function
such that $\supp\tz\subset B_0$,
$\tz(x)=1$ for all $x\in B_{-1}$, and
$\|\partial^\az\tz\|_{L^\infty(\rn)}\le 1$ for $|\az|\le N$. Obviously, $\tz$ is
an $N$-normalized bump function associated to $B_0$.
For all $x\in\rn$ and
$k\in\nn$, set $D_0(x)\equiv \psi(x)\tz(x)$ and $D_k(x)\equiv
\psi(x)[\tz(A^{-k}x)-\tz(A^{-(k-1)}x)]$. It is easy to check that
$\psi(x)=\sum_{k=0}^\infty D_k(x)$ pointwise.  For any $k\in\zz_+$,
let $d_k\equiv \int_\rn D_k(x)\,
dx$, $s_0\equiv 0$ and $s_k=\sum_{j=0}^{k-1}d_j$ for $k\ge 1$.

Notice that for any $k\in\nn$, we have $\supp D_k\subset B_{k}
\setminus B_{k-2}$. Fix $M>0$. Since $D_k(x)\ne 0$ implies
that $\rho(x)\sim b^{k}$, we have
\begin{equation}\label{e1}
|D_k(x)|\ls
[\rho(x)]^{-M-1}\ls b^{-(M+1)k},
\end{equation}
due to the fact that $\psi\in\cs_0(\rn)$ and $\|\tz\|_{L^\infty(\rn)}\le 1$.
From this and
$\supp D_k\subset B_{k}$, it follows that
$\sum_{k=0}^\fz\int_\rn |D_k(x)|\, dx\ls 1$. Using that
$\psi=\sum_{k=0}^\fz D_k$ and $\psi\in\cs_0(\rn)$, we
obtain $\sum_{k=0}^\fz d_k=\int_\rn \psi(x)\, dx=0.$ Thus, we also
have $s_k=-\sum_{j\ge k}d_j$. Moreover, from \eqref{e1} and
$\supp D_k\subset B_{k}$, it follows that $|d_k|\ls
b^{-kM}$, and hence $|s_k|\ls b^{-kM}$.

For any  $k\in\zz_+$ and $x\in\rn$, we define
\[
\tl D_k(x)\equiv D_k(x)-d_kb^{-k}\tl\tz(A^{-k}x)+s_k
[b^{-(k-1)}\tl\tz(A^{-(k-1)}x)- b^{-k}\tl\tz(A^{-k}x)],
\]
where $\tl\tz(x)=\tz(x)/\|\tz\|_{L^1(\rn)}$.
We claim that $\psi^{(k)}\equiv b^{Mk}\wz
D_k\in\cs_0(\rn)$ is the desired constant multiple of an
$N$-normalized bump function associated to $B_k$. Indeed, it is easy
to check that $\wz D_k\in\ccc^\fz(\rn)$ with $\supp \tl D_k\subset
B_k$, $\int_\rn \tl D_k(x)\,dx=0$. Using $\sum_{k=0}^\fz
d_k=0$ and $s_k=-\sum_{j\ge k}d_k$, by Abel's summation, we have
\[
\sum_{k=0}^\fz
s_k[b^{-(k-1)}\tl\tz(A^{-(k-1)}x)-b^{-k}\tl\tz(A^{-k}x)]=\sum_{k=0}^\fz
d_kb^{-k}\tl\tz(A^{-k}x).
\]
This together with $\psi=\sum_{k=0}^\infty D_k$
implies that $\psi=\sum_{k=0}^\fz \tl D_k=\sum_{k=0}^\fz
b^{-Mk}\psi^{(k)}$.

Finally, it remains to show that
$\|\partial^\az \wz D_k(A^k\cdot)\|_{L^\infty(\rn)}\ls b^{-Mk}$ for any
$k\in\zz_+$ and $|\az|\le N$. Since $\|\partial^\az\tz\|_{L^\infty(\rn)},\
\|\partial^\az\tz(A\cdot)\|_{L^\infty(\rn)}\ls 1$,
and $|s_k|,\,|d_k|\ls b^{-Mk}$, it suffices to prove
$$\|\partial^\az D_k(A^k\cdot)\|_{L^\infty(\rn)}\ls b^{-Mk}.$$
Recall that $\supp D_k \subset B_{k} \setminus B_{k-2}$ for $k\in\nn$.
Thus, we only need to check
$|\partial^\az D_k(A^k\cdot)(x)|\ls b^{-Mk}$
for all $x \in B_0\setminus B_{-2}$ for all $|\az|\le N$.
Since $\psi\in\cs(\rn)$,
for all $x \in B_0\setminus B_{-2}$ and for all $|\az|\le N$, we have
\begin{eqnarray}\label{e3.x0}
 |\partial^\az D_k(A^k\cdot)(x)|
&& =|\partial^\az [\psi(A^k\cdot)(\theta(\cdot)-\theta(A^{-1}\cdot)](x)|\\
&&\ls \sum_{|\bz|\le|\az|}|\partial^\bz \psi(A^k\cdot)(x)|
\ls \|A^k\|^{|\alpha|} \sum_{|\bz|\le|\az|}|\partial^\bz \psi(A^k x)|
\nonumber \\
&& \ls \|A\|^{Nk}\rho(A^k x)^{-M'}
\ls \|A\|^{Nk} b^{-kM'} \ls b^{-kM}.\nonumber
\end{eqnarray}
This finishes the proof of
the claim and hence Lemma \ref{l3.1}.
\end{proof}

For $i=1,2$, let $A_i$ be a dilation on $\rr^{n_i}$ as in Definition \ref{d1.1}. Let
$\lz_{i,\,-}$ and $\lz_{i,\,+}$ be two {\it positive numbers} such
that
$$1<\lz_{i,\,-}<\min\{|\lz|: \lz\in\sz(A_i)\}\le\max\{|\lz|:
\lz\in\sz(A_i)\}<\lz_{i,\,+}.$$
In the case when $A_i$ is diagonalizable over $\mathbb C$,
we can even take $\lz_{i,\,-}=\min$ and $\lz_{i,\,+}=\max$ above.
Otherwise, we need to choose them sufficiently close to these
equalities according to what we need in our arguments. Let
$\zeta_{i,\,\pm}=\log_{b_i}\lz_{i,\,\pm}$.
It is useful to make some remarks about Definitions \ref{d1.3} and \ref{d1.3A}.

\begin{rem}\rm\label{r1.1}
(i) One can show that if $T$ is a PASIO of order $(1,1)$, then it
is also a PASIO of order $0$. In fact, this is a consequence of Lemma
\ref{l3.2} below.

(ii) In Definition \ref{d1.3A}, the range of $\ez_i$ is
effectively restricted to the interval $(0,\,\log_{b_i}|\lz_{i,\,+}|]$,
where $\lz_{i,\,+}$ denotes the
largest eigenvalue of $A_i$ in absolute value. Again we
will see this in the one parameter setting.
In fact, assume that $\ez>\log_b|\lz_+|$ and
$|K(x)|\le C_1[\rho(x)]^{-1}$ and $|K(x+h)-K(x)|\le C_1[\rho(h)]^\ez
[\rho(x)]^{-1-\ez}$ for $\rho(h)\le b^{-2\sz}\rho(x)$ and $x\ne0$.
Choose $\lz $ such that $|\lz_+|<\lz <b^\ez$ and let
$\zeta \equiv \log_b\lz $. For any $x\ne0$, when
$\rho(h)\le\min\{1,\,b^{-2\sz}\rho(x)\}$, by
$[\rho(h)]^{\zeta }\le C_1|h|$ (see \cite[(3.3)]{b1}), we have
$$|K(x+h)-K(x)|\le C_1[\rho(x)]^{-1-\ez} |h|^{\ez/\zeta }\le C_1
[\rho(x)]^{-1-\ez} |h|,$$ which implies that $K$ is locally Lipschitz
continuous away from $0$. Moreover,  for all $x\ne0$,
$$\limsup_{h\to0}\frac1{|h|}|K(x+h)-K(x)|
\le C_1[\rho(x)]^{-1-\ez}\limsup_{h\to0}[\rho(h)]^{\ez-\zeta }=0,$$
which implies that $K$
is a constant function away from $0$ and thus,
by $|K(x)|\le C_1[\rho(x)]^{-1}$, we further have $K(x)=0$ for all $x\ne0$.
\end{rem}

\begin{lem}\label{l3.2}
Let $K$ be the kernel of a PASIO of order $(s_1+1,s_2+1)$,
where $s_1,s_2 \in \zz_+$. Then, there exist positive constants
$C$ and $\epsilon_i$ such that $K$ has the following 3 additional properties:

$(K1')$ for all $(x_1,\,x_2)\in\boz_{n\times m}$ with $\rho_1(x_1)=b_1^{\ell_1}$
for certain $\ell_1\in\zz$,
$h_1\in\rn$ with $\rho_1(h_1)\le b_1^{-2\sz_1}\rho_1(x_1)$ and
$\az_1\in\zz_+^n$ with $|\az_1|=s_1$,
\begin{eqnarray*}
 \bigg|\Delta_{A_1^{-\ell_1}h_1}^{(1)}
\partial^{\az_1}_1[K(A_1^{\ell_1}\cdot,
\, x_2)](A_1^{-\ell_1}x_1)\bigg| \le
C_1\frac{[\rho_1(h_1)]^{\ez_1}}{[\rho_1(x_1)]^{1+\ez_1}}\frac
1{\rho_2(x_2)}.
\end{eqnarray*}
This also holds with the roles of $x_1$ and $x_2$ interchanged;

$(K1'')$ for all $(x_1,\,x_2)\in\boz_{n\times m}$ with $\rho_i(x_i)=b_i^{\ell_i}$
 for certain $\ell_i\in\zz$,
$h_i\in\rr^{n_i}$ with $\rho_i(h_i)\le
b_i^{-2\sz_i}\rho_i(x_i)$ and $\az_i\in\zz_+^{n_i}$ with $|\az_i|=
s_i$, $i=1,\,2$,
\begin{eqnarray*}
&&\bigg|\Delta_{A_1^{-\ell_1}h_1}^{(1)}\Delta^{(2)}_{A_2^{-\ell_2}h_2}
\partial^{\az_1}_1\partial^{\az_2}_2[K(A_1^{\ell_1}\cdot,
\, A_2^{\ell_2}\cdot)](A_1^{-\ell_1}x_1,\,
A_2^{-\ell_2}x_2)\bigg|\le C_1
\frac{[\rho_1(h_1)]^{\ez_1}}{[\rho_1(x_1)]^{1+\ez_1}}
\frac{[\rho_2(h_2)]^{\ez_2}}{[\rho_2(x_2)]^{1+\ez_2}};
\end{eqnarray*}

$(K3')$ the kernel $K^{\psi^{(2)},\,k_2}$ as in (K3)
satisfies that
for all $x_1\in\rr^n\setminus\{0\}$,
$ h_1\in\rr^n$ with $\rho_1(h_1)\le b_1^{-2\sz_1}\rho_1(x_1)$,
and $\az_1\in\zz_+^n$ with $|\az_1|=s_1$,
\begin{eqnarray*}
&&\lf|\Delta^{(1)}_{A_1^{-\ell_1}h_1}\partial_1^{\az_1}
[K^{\psi^{(2)},\,k_2}(A_1^{\ell_1}\cdot)](A_1^{-\ell_1}x_1)\r| \le
C_1 \frac{[\rho_1(h_1)]^{\ez_1}}{[\rho_1(x_1)]^{1+\ez_1}}.
\end{eqnarray*}
This also holds with the roles of $x_1$ and $x_2$ interchanged.
\end{lem}

\begin{proof}
We will only prove (K3${}'$). Other properties are shown in the
same fashion. Let $K\equiv K^{\psi^{(2)},k_2}$ be the kernel as in (K3).
Assume that $\rho(x_1)=b_1^{\ell_1}$ and $\rho(h_1)
\le b_1^{-2\sigma_1 +\ell_1}$. Take any $\ez_1\in (0,\,
\log_{b_1}\lz_{1,-})$. By (K3) and the Taylor's formula,
for $|\alpha_1|=s_1$ we have
\begin{eqnarray*}
|\Delta_{A_1^{-\ell_1}h_1} \partial^{\alpha_1}_1
[K(A^{\ell_1}\cdot)](A_1^{-\ell_1}x_1)|
&& \le |A_1^{-\ell_1}h_1| \sup_{\rho(z_1)
\le b_1^{-2\sz_1 +\ell_1} }|\nabla\partial^{\alpha_1}_1
[K(A_1^{\ell_1}\cdot)](A_1^{-\ell_1}(x_1+z_1))|\\
&&\ls |A_1^{-\ell_1}h_1| b_1^{-\ell_1}
\ls [\rho(h_1)]^{\ez_1} [\rho(x_1)]^{-(1+\ez_1)},
\end{eqnarray*}
which completes the proof.
\end{proof}

The following lemma plays a key role in the proof of Theorems
\ref{t1.1} and \ref{t1.2} by generalizing \cite[Lemma
1]{hy05} to the anisotropic setting and to the higher order partial
derivatives of corresponding kernels.

\begin{lem}\label{l3.2A}
Let $K$ be the kernel of a PASIO of order $(s_1+1,s_2+1)$,
where $s_1,s_2 \in \zz_+$. For $i=1,\,2$, let
$\vz^{(i)}\in\cs_0(\rr^{n_i})$ be an $(N_i+s_i+1)$-normalized
bump function associated to some dilated ball $B^{(i)}_{j_i}$,
where $j_i\in\zz_+$ and $N_i\in\nn$ is as in (K2) and (K3) of Definition \ref{d1.3}.
For all $k_1,\,k_2\in\zz$, define
$K_{k_1,\,k_2}\equiv K\ast\vz_{k_1,\,k_2}$, where
$\vz\equiv\vz^{(1)}\otimes \vz^{(2)}$.
Then, there exist
positive constants $C$ and $\ez_i$ such that:
for all $h_i,\,x_i,\,y_i\in\rr^{n_i}$ with
$\rho_i(x_i)=b_i^{\ell_i}$ for certain $\ell_i\in\zz$ and
$\rho_i(x_i-y_i)<b_i^{k_i}$, and $\az_i\in\zz_+^{n_i}$ with
$|\az_i|\le s_i$, $i=1,\,2$,
$$\ |\partial^{\az_1}_1\partial^{\az_2}_2
[K_{k_1,\,k_2}(A_1^{\ell_1}\cdot,\,A_2^{\ell_2}\cdot )](
A_1^{-\ell_1}y_1,\,A_2^{-\ell_2}y_2)|\le
C\prod_{i=1}^2\dfrac{b_i^{k_i\ez_i}}{[b_i^{k_i}
+b_i^{-j_i}\rho_i(x_i)]^{1+\ez_i}}.$$
\end{lem}

\begin{proof}
To prove this lemma, we first present two basic facts. Let
$i=1,\,2$. For any $\az_i\in\zz_+^{n_i}$, by (3.13) in \cite{bh}
when $\ell_i-k_i<0$ or a similar proof when $\ell_i-k_i\ge 0$, for
all $x_i,\, z_i\in\rr^{n_i}$, we have
\begin{eqnarray}\label{e3.1}
\quad\partial^{\az_i}[\vz^{(i)}(A_i^{\ell_i-k_i}\cdot-
A_i^{-k_i}z_i)](A_i^{-\ell_i}y_i)&&=
\partial^{\az_i}[\vz^{(i)}(A_i^{\ell_i-k_i}\cdot)](A_i^{-\ell_i}(y_i-z_i))\\
&&=\sum_{|\bz_i|=|\az_i|}a^{(i)}_{\bz_i}
\partial^{\bz_i}[\vz^{(i)}(A_i^{j_i}\cdot)](A_i^{-j_i-k_i}(y_i-z_i)),\nonumber
\end{eqnarray}
where
\begin{eqnarray}\label{e3.2}
|a^{(i)}_{\bz_i}|\ls b_i^{(\ell_i-j_i-k_i)|\bz_i|\zeta_{i,\,-}}\hs
{\rm if} \hs \ell_i-j_i-k_i\le 0,
\end{eqnarray}
and
\begin{eqnarray}\label{e3.3}
|a^{(i)}_{\bz_i}|\ls b_i^{(\ell_i-j_i-k_i)|\bz_i|\zeta_{i,\,+}}\hs
{\rm if} \hs\ell_i-j_i-k_i>0.
\end{eqnarray}
%
Moreover,  for any fixed  $x_i\in\rr^{n_i}$ with
$\rho_i(x_i)=b_i^{\ell_i}$, if $\ell_i\le k_i+j_i+4\sz_i$, we claim
that

\begin{eqnarray}\label{e3.4}
\xi^{(i)}_{\bz_i}(z_i)\equiv
\partial^{\bz_i}[\vz^{(i)}(A_i^{j_i}\cdot)](A_i^{-j_i-k_i}y_i-A_i^{6\sz_i+1}z_i)
\end{eqnarray}
is an $N_i$-normalized bump function associated to $B^{(i)}_0$.
Indeed, if $\xi^{(i)}_{\bz_i}(z_i)\ne 0$, then by
$\supp(\partial^{\bz_i}\vz^{(i)}) \subset B^{(i)}_{j_i}$, $x_i\in
B^{(i)}_{\ell_i+1},\, y_i\in x_i+B_{k_i+1}^{(i)},\, \ell_i \le
k_i+j_i+4\sz_i$ and \eqref{e2.2}, we obtain
\begin{eqnarray*}
z_i\in A_i^{-k_i-j_i-6\sz_i-1}y_i+B^{(i)}_{-6\sz_i-1}\subset
B^{(i)}_{-\sz_i}+B^{(i)}_{-6\sz_i-1}\subset B^{(i)}_0.
\end{eqnarray*}
Moreover, since $\vz^{(i)}$ is an $(s_i+N_i+1)$-normalized bump
function associated to $B^{(i)}_{j_i}$, then for all
$z_i\in\rr^{n_i}$ and $\gz_i\in\zz_+^{n_i}$ with $|\gz_i|\le N_i$,
we have $ |\partial^{\gz_i}(\xi^{(i)}_{\bz_i})(z_i)|\ls 1. $
Thus, the above claim holds.

We now show Lemma \ref{l3.2A} by considering the following four cases. In the
following {\it Case (i)} through {\it Case (iv)}, we always assume that
$\rho_i(x_i)=b_i^{\ell_i}$ for certain $\ell_i\in\zz$ and
$\az_i\in\zz_+^{n_i}$ with $|\az_i|\le s_i$, $i=1,\,2$.

{\it Case  (i).} $\ell_1\le k_1+j_1+4\sz_1$ and $\ell_2\le
k_2+j_2+4\sz_2$. In this case, by \eqref{e3.1}, \eqref{e3.2}, \eqref{e3.3},
\eqref{e3.4}, (K2), $|\az_i|\le s_i$,
$\zeta_{i,\,+}=\log_{b_i}\lz_{i,\,+}<1$ and $j_i\ge 0$, we have
\begin{eqnarray*}
&&|\partial^{\az_1}_1\partial^{\az_2}_2
[K_{k_1,\,k_2}(A_1^{\ell_1}\cdot,\,A_2^{\ell_2}\cdot )](
A_1^{-\ell_1}y_1,\,A_2^{-\ell_2}y_2)|\\
&&\hs=\lf|b_1^{-k_1}b_2^{-k_2}\sum_{\gfz{|\bz_1|\le |\az_1|}
{|\bz_2|\le |\az_2|}}a^{(1)}_{\bz_1}a^{(2)}_{\bz_2}\lf\langle K,\,
 \bigotimes_{i=1}^2\xi^{(i)}_{\bz_i}
(A_i^{-j_i-k_i-6\sz_i-1}\cdot)\r\rangle\r|\ls b_1^{-k_1}b_2^{-k_2},
\end{eqnarray*}
which is desired. Here
\begin{eqnarray*}
&&\bigotimes_{i=1}^2 \xi^{(i)}_{\bz_i}
(A_i^{-j_i-k_i-6\sz_i-1}\cdot)\equiv \xi^{(1)}_{\bz_1}
(A_1^{-j_1-k_1-6\sz_1-1}\cdot)\bigotimes\xi^{(2)}_{\bz_2}
(A_2^{-j_2-k_2-6\sz_2-1}\cdot).
\end{eqnarray*}

{\it Case   (ii).} $\ell_1\le k_1+j_1+4\sz_1$ and $\ell_2>
k_2+j_2+4\sz_2$. In this case, if $z_2\in B^{(2)}_{k_2+j_2}$,
$\rho_2(x_2-y_2)<b_2^{k_2}$ and $x_2\in B^{(2)}_{\ell_2+1}\setminus
B^{(2)}_{\ell_2}$ with $\ell_2> k_2+j_2+4\sz_2$, then by Definition
\ref{d2.1}, it is easy to obtain that $\rho_2(y_2)\ge
b^{\ell_2-\sz_2}$ and $\rho_2(z_2)< b_2^{-3\sz_2}\rho_2(y_2)$. Thus, by
$\vz^{(2)}\in\cs_0(\rrm)$, \eqref{e3.1}, \eqref{e3.2}, \eqref{e3.3}
and \eqref{e3.4} with $i=1$, $\supp\vz^{(2)}(A_2^{-k_2}\cdot)\subset
B^{(2)}_{j_2+k_2}$ and (K$3'$),  we have
\begin{eqnarray*}
&&|\partial^{\az_1}_1\partial^{\az_2}_2
[K_{k_1,\,k_2}(A_1^{\ell_1}\cdot,\,A_2^{\ell_2}\cdot )](
A_1^{-\ell_1}y_1,\,A_2^{-\ell_2}y_2)|\\
&&\ =\Bigg|b_1^{-k_1}\sum_{|\bz_1|\le
|\az_1|}a^{(1)}_{\bz_1}\int_{\rrm}
\vz^{(2)}_{k_2}(z_2)\Delta^{(2)}_{-A_2^{\ell_2}z_2}
\partial^{\az_2}_2[K^{ \xi^{(1)}_{\bz_1} ,\,-j_1-k_1-6\sz_1-1}(A_2^{\ell_2}\cdot)]
(A_2^{-\ell_2}y_2)\, dz_2\Bigg|\\
&&\ \ls b_1^{-k_1}\int_{B^{(2)}_{j_2+k_2}}
\frac{[\rho_2(z_2)]^{\ez_2}}
{[\rho_2(x_2)]^{1+\ez_2}}|\vz^{(2)}_{k_2}(z_2)|\, dz_2
\ls b_1^{-k_1}b_2^{j_2+(j_2+k_2)\ez_2-\ell_2(1+\ez_2)},
\end{eqnarray*}
which is desired.

{\it Case  (iii). } $\ell_1> k_1+j_1+4\sz_1$ and $\ell_2\le
k_2+j_2+4\sz_2$. In this case, by symmetry, similarly to the estimate of {\it Case
(ii)}, we also have
$$|\partial^{\az_1}_1\partial^{\az_2}_2
[K_{k_1,\,k_2}(A_1^{\ell_1}\cdot,\,A_2^{\ell_2}\cdot )](
A_1^{-\ell_1}y_1,\,A_2^{-\ell_2}y_2)|\ls
b_1^{j_1+(j_1+k_1)\ez_1-\ell_1(1+\ez_1)}b_2^{-k_2}.$$

{\it Case  (iv). } $\ell_1> k_1+j_1+4\sz_1$ and $\ell_2>k_2+j_2+4\sz_2$.
In this case, for $i=1,\,2$, $z_i\in B^{(i)}_{j_i+k_i}$, $\rho_i(x_i-y_i)<
b^{k_i}$ and $\rho_i(x_i)=b_i^{\ell_i}$,  we have $\rho_i(y_i)\ge
b_i^{\ell_i-\sz_i}$ and $\rho_i(z_i)< b_i^{-3\sz_i}\rho_i(y_i)$. By
this, $\vz^{(i)}\in\cs_0(\rr^{n_i})$,
$\supp\vz^{(i)}(A_i^{-k_i}\cdot)\subset B^{(i)}_{j_i+k_i}$ and (K$1''$), we obtain
\begin{eqnarray*}
&&|\partial^{\az_1}_1\partial^{\az_2}_2
[K_{k_1,\,k_2}(A_1^{\ell_1}\cdot,\,A_2^{\ell_2}\cdot )](
A_1^{-\ell_1}y_1,\,A_2^{-\ell_2}y_2)|\\
&&\hs=\Bigg|\int_{\rnm}\vz^{(1)}_{k_1}(z_1)\vz^{(2)}_{k_2}(z_2)\\
&&\hs\hs\times \Delta^{(1)}_{-A_1^{-\ell_1}z_1}
\Delta^{(2)}_{-A_2^{-\ell_2}z_2}\partial^{\az_1}_1
\partial^{\az_2}_2[
K(A_1^{\ell_1}\cdot,\,A_2^{\ell_2}\cdot)](A_1^{-\ell_1}y_1
,\,A_2^{-\ell_2}y_2)\,dz\Bigg|\\
&&\hs\ls\int_{B^{(1)}_{j_1+k_1}\times
B^{(2)}_{j_2+k_2}}|\vz^{(1)}_{k_1}(z_1)\vz^{(2)}_{k_2}(z_2)|
\frac{[\rho_1(z_1)]^{\ez_1}}{[\rho_1(x_1)]^{1+\ez_1}}
\frac{{[\rho_2(z_2)]^{\ez_2}}}
{[\rho_2(x_2)]^{1+\ez_2}}\,dz
\ls \prod_{i=1}^2b_i^{j_i+\ez_i(j_i+k_i)-\ell_i(1+\ez_i)},
\end{eqnarray*}
which is desired.

Combining the above estimates completes the proof of Lemma \ref{l3.2A}.
\end{proof}

\begin{rem}\label{r3.2}
\rm Notice that in the proof of Lemma \ref{l3.2A} we have not
used explicitly the bounds on the highest order derivatives of $K$.
Instead, we used the difference properties (K$1'$), (K$1''$),
and (K$3'$) from Lemma \ref{l3.2}. Thus, if $K$ is merely a kernel
of a PASIO of order 0, then conclusions of Lemma \ref{l3.2A} apply.
In particular, there exists a
positive constant $C$ such that for all $h_i,\,x_i,\,y_i\in\rr^{n_i}$ with
$\rho_i(x_i)=b_i^{\ell_i}$ for certain $\ell_i\in\zz$ and
$\rho_i(x_i-y_i)<b_i^{k_i}$,
\begin{equation}\label{e33}
\ |
K_{k_1,\,k_2}(y_1,\,y_2)|\le
C\prod_{i=1}^2\dfrac{b_i^{k_i\ez_i}}{[b_i^{k_i}
+b_i^{-j_i}\rho_i(x_i)]^{1+\ez_i}}.
\end{equation}
\end{rem}

\begin{proof}[Proof of Theorem \ref{t1.1}]
Let $p\in(1,\,\fz)$ and $w\in\wwp$.
Let $T$ be a product anisotropic singular integral
operator (PASIO) of order $0$ with kernel $K$
as in Definition \ref{d1.3A}. Let $N_1$ and $N_2$  be
as in (K2) and (K3). Let
$\psi\equiv\psi^{(1)} \otimes \psi^{(2)}$ and
$\phi\equiv\phi^{(1)} \otimes \phi^{(2)}$ be
as in Proposition \ref{p2.1}. By part (iv) of this proposition, we have
$\psi=\phi\ast\phi$. From Proposition \ref{p2.2},
it follows that for all $f\in\cd(\rnm)$,
\begin{eqnarray}\label{e3.7}
\|\vec S_{\phi\ast\phi}(f)\|_{L^p_w(\rnm)}\sim\|f\|_{L^p_w(\rnm)}.
\end{eqnarray}

Since $\phi^{(i)}\in\cs_0(\rr^{n_i})$ for $i=1,\,2$, then by
Lemma \ref{l3.1}, we have
$\phi^{(i)}=\sum_{j_i=0}^\fz b_i^{-4j_i}\phi^{(i,\,j_i)},$ where
$\phi^{(i,\,j_i)}\in\cs_0(\rr^{n_i})$ is a constant multiple of an
$(N_1+1)$-normalized bump function associated to $B^{(i)}_{j_i}$.
For $j_1,\, j_2\in\zz_+$ and
$k_1,\,k_2\in\zz$, let
$\phi^{\{j_1,\,j_2\}}\equiv\phi^{(1,\,j_1)}\otimes\phi^{(2,\,j_2)}$ and
$K_{k_1,\,k_2}^{j_1,\,j_2}\equiv
K\ast\phi^{\{j_1,\,j_2\}}_{k_1,\,k_2}$. For any $x,\,z\in\rnm$,
$k_1,\,k_2\in\zz,\,j_1,\,j_2\in\zz_+$, $y\in\rnm$ with
$\rho_1(y_1)<b_1^{k_1}$ and $\rho_2(y_2)<b_2^{k_2}$, and
locally integrable function $f$ on $\rnm$, by the estimate
\eqref{e33} in Remark \ref{r3.2}
and $b_i^{k_i}+b_i^{-j_i}\rho_i(z_i)\sim b_i^{k_i}+b_i^{-j_i}\rho_i(z_i-y_i)$,
$i=1,\,2$, we obtain
\begin{eqnarray}\label{e3.8}
&|f\ast K^{j_1,\,j_2}_{k_1,\,k_2}(x-y)|&\ls
\int_{\rnm}|f(x-y-z)|\prod_{i=1}^2
\frac{b_i^{k_i\ez_i}}{[b_i^{k_i}+b_i^{-j_i}\rho_i(z_i)]^{1+\ez_i}}\,dz\\
&&\ls
\int_{\rnm}|f(x-z)|\prod_{i=1}^2
\frac{b_i^{k_i\ez_i}}{[b_i^{k_i}+b_i^{-j_i}\rho_i(z_i)]^{1+\ez_i}}\,dz \nonumber\\
&&\ls
b_1^{j_1(1+\ez_1)}b_2^{j_2(1+\ez_2)}\cm_s(f)(x),\nonumber
\end{eqnarray}
where and in what follows,  $\cm_s(f)$ denotes the {\it strong
maximal function} which is defined by
setting, for all $x\in\rnm$,
$$\cm_s(f)(x)\equiv\sup_{k_1,\,k_2\in\zz}\sup_{x\in y+B^{(1)}_{k_1}
\times B^{(2)}_{k_2}} \frac{1}{b_1^{k_1}b_2^{k_2}}
\int_{y+B^{(1)}_{k_1}\times B^{(2)}_{k_2}}|f(z)|\,dz.$$
Thus, by \eqref{e3.7}, \eqref{e3.8}, the weighted
vector-valued maximal inequality for $\cm_s$ (see \cite[Proposition
2.2]{blyz1}), and the $L^p_w(\rnm)$-boundedness of $\vec
g_{\phi}$ which was proved in the proof of
\cite[Theorem 3.2]{blyz1}, we have that for $f\in \cd(\rnm)$,
\begin{eqnarray*}
&&\|Tf\|_{L^p_w(\rnm)}\\
&&\hs\ls \sum_{j_1=0}^\fz\sum_{j_2=0}^\fz b_1^{-4j_1}
b_2^{-4j_2}\\
&&\hs\hs\times\lf\|\lf\{\sum_{k_1,\,k_2\in\zz}\frac1{b_1^{k_1}b_2^{k_2}}
\int_{B^{(1)}_{k_1}\times
B^{(2)}_{k_2}}|K^{j_1,\,j_2}_{k_1,\,k_2}\ast
f\ast\phi_{k_1,\,k_2}(\cdot-y)|^2\, dy\r\}^{1/2}
\r\|_{L^p_w(\rnm)}\\
&&\hs\ls\sum_{j_1=0}^\fz\sum_{j_2=0}^\fz b_1^{-2j_1 }
b_2^{-2j_2}\lf\|\lf\{\sum_{k_1,\,k_2\in\zz}|
\cm_s\lf(f\ast\phi_{k_1,\,k_2}\r)|^2\r\}^{1/2}\r\|_{L^p_w(\rnm)}\\
&&\hs\ls\sum_{j_1=0}^\fz\sum_{j_2=0}^\fz b_1^{-2j_1} b_2^{-2j_2
}\lf\|\vec g_\phi(f)\r\|_{L^p_w}\ls
\|f\|_{L^p_w(\rnm)}.
\end{eqnarray*}
This combined with the density of $\cd(\rnm)$ in $L^p_w(\rnm)$ then
completes the proof of Theorem \ref{t1.1}.
\end{proof}

\section{Proof of Theorem
\ref{t1.2}}\label{s4}

To prove Theorem \ref{t1.2}, we need to use a vector-valued variant
of the boundedness criterion established in
\cite[Corollary 6.5]{blyz1}. We shall use an analogue of the
grid of Euclidean dyadic cubes which is mainly due to Christ \cite{c90}
and formulated as in \cite[Lemma 2.2]{blyz1}.

\begin{lem}\label{l3.3} Let $A$ be a dilation. There exists a
collection
$\cq\equiv\{Q^k_\az\subset \rn:\ k\in\zz,\ \az\in
\mathrm I_k\}$ of open subsets, where $\mathrm I_k$ is certain index set, such that

(i) $|\rn\setminus\cup_\az Q^k_\az|=0$ for each fixed $k$ and
$Q^k_\az\cap Q^k_\bz=\emptyset$ if $\az\not=\bz$;

(ii) for any $\az,\, \bz,\, k,\, \ell$ with $\ell\ge k$, either
$Q^k_\az\cap Q^\ell_\bz=\emptyset$ or $Q^\ell_\az\subset Q^k_\bz$;

(iii) for each $(\ell,\, \bz)$ and each $k<\ell$, there exists a
unique $\az$ such that $Q^\ell_\bz\subset Q^k_\az$;

(iv) there exist certain negative integer $v$ and positive integer
$u$ such that for all $Q^k_\az$ with $k\in\zz$ and $\az\in \mathrm I_k$,
there exists $x_{Q^k_\az}\in Q^k_\az$ satisfying that for all $x\in
Q^k_\az$,
$x_{Q^k_\az}+B_{vk-u}\subset Q^k_\az\subset x+B_{vk+u}.$
\end{lem}

In what follows, for convenience, we call $\{Q^k_\az\}_{k\in\zz,\,
\az\in \mathrm I_k}$ dyadic cubes. Also for any dyadic cube
$Q^k_\az$ with $k\in\zz$ and $\az\in \mathrm I_k$, we always set
$\ell(Q^k_\az)\equiv k$ as its {\it level}.

Let $A_i$ be a dilation on $\rr^{n_i}$, and $\cq^{(i)}$,
$\ell(Q_i),\,v_i,\,u_i$ the same as in Lemma \ref{l3.3}
corresponding to $A_i$ for $i=1,\,2$. Let $\hr\equiv\cq^{(1)}\times
\cq^{(2)}$. For $R\in\hr$, we always write $R\equiv R_1\times R_2$ with
$R_i\in\cq^{(i)}$ and call $R$ a dyadic rectangle. We need the notion
of rectangular atoms for anisotropic product Hardy
spaces.

\begin{defn}\label{d3.1}\rm
Let $w\in \wfz$ and $q_w$ be as in \eqref{e2.4}. The triplet $(p,\
q,\ \vec s)_w$ is called {\it admissible} if $p\in(0,\, 1]$,
$q\in [2,\, \fz)\cap(q_w,\, \fz)$ and $\vec s\equiv (s_1,\,s_2)$
with $s_i\in\zz_+$ and $s_i\ge\lfloor
(\frac{q_w}p-1)\zeta_{i,-}^{-1}\rfloor$, $i=1,\ 2$. For any
$R\in\hr$, a function $a_R$ is called a {\it rectangular $(p,\,q,\,\vec
s)_w$-atom} if

(i) $a_R$ is supported on $R''=R''_1\times R''_2$, where
$R''_i\equiv x_{R_i}+B^{(i)}_{v_i(\ell(R_i)-1)+u_i+3\sz_i},\ i=1,\
2$;

(ii) $\int_\rn a_R(x_1,\ x_2)x_1^\az\,dx_1=0$ for all $|\az|\le s_1$
and almost all $x_2\in \rrm$, and

\quad\ \  $\int_\rrm a_R(x_1,\, x_2)x_2^\bz\,dx_2=0$ for all
$|\bz|\le s_2$ and almost all $ x_1\in \rn$;

(iii) $\|a\|_{L^q_w(\rnm)}\le [w(R)]^{1/q-1/p}$.
\end{defn}

We also need to consider the vector-valued space
\begin{eqnarray*}
&&\ch\equiv \{\{f_{k_1,\,k_2}\}_{k_1,\,k_2\in\zz}: \,f_{k_1,\,k_2} \
{\rm is\ a\
measurable\ function\ on } \ B^{(1)}_{k_1}\times B^{(2)}_{k_2}\\
&&\hs \hspace{6cm} {\rm \ for \ any} \ k_1,\,k_2\in\zz\ {\rm and}\
|\{f_{k_1,\,k_2}\}_{k_1,\,k_2\in\zz}|_{\ch}<\fz\},
\end{eqnarray*}
where
\begin{eqnarray*}
|\{f_{k_1,\,k_2}\}_{k_1,\,k_2\in\zz}|_{\ch}\equiv
\bigg\{\sum_{k_1,\,k_2\in\zz}b^{-k_1}_1b_2^{-k_2}
\int_{B^{(1)}_{k_1}\times B^{(2)}_{k_2}} |f_{k_1,\,k_2}(y)|^2\, dy
\bigg\}^{1/2}.
\end{eqnarray*}
In what follows, for $x\in\rnm$, we always write
\begin{eqnarray*}
|\{f_{k_1,\,k_2}(x)\}_{k_1,\,k_2\in\zz}|_{\ch}\equiv
\bigg\{\sum_{k_1,\,k_2\in\zz}b^{-k_1}_1b_2^{-k_2}
\int_{B^{(1)}_{k_1}\times B^{(2)}_{k_2}} |f_{k_1,\,k_2}(x-y)|^2\, dy
\bigg\}^{1/2}.
\end{eqnarray*}

Finally, let $p\in(0,\,\fz)$ and $w\in\wfz$.
Define the vector-valued space $L^p_{w,\,\ch}(\rnm)$
as the collection of all sequences $\{f_{k_1,\,k_2}
\}_{k_1,\,k_2\in\zz}$ of measurable functions on $\rnm$ with the norm
\begin{eqnarray*}
\|\{f_{k_1,\,k_2}\}_{k_1,\,k_2\in\zz}\|_{L^p_{w,\,\ch}(\rnm)}
\equiv\lf\{\int_\rnm\lf|\{f_{k_1,\,k_2}(x)\}_{k_1,\,k_2\in\zz}\r|^p_\ch
w(x)\,dx\r\}^{1/p}<\infty.
\end{eqnarray*}

The following conclusion is the vector-valued variant of
 \cite[Corollary 6.5]{blyz1}, whose proof is similar to that of
\cite[Corollary 6.1]{blyz1}; see also \cite[Corollary 1.1]{cyz}
for the corresponding result on $H^p(\rnm)$. Here we omit the details.

\begin{lem}\label{l3.4}
Let $(p,\,q_1,\,\vec s)_w$ be an
admissible triplet as in Definition \ref{d3.1}. Let $q_0\in[q_1,\,
\fz)$ and $\{T_{k_1,\,k_2}\}_{k_1,\,k_2\in\zz}$ be an $\ch$-valued
linear operator bounded from $L^{q_1}_w(\rnm)$ to
$L^{q_0}_{w,\,\ch}(\rnm)$. Let $q\in [p,\, 2)$ be such that
$1/q-1/p=1/q_0-1/q_1$.

Suppose that there exist positive constants $C,\, \ez$
such that for all $\gz\in\zz_+$ and all rectangular $(p,\,
q_1,\,\vec s)_w$-atoms $a_R$,
\begin{eqnarray*}
\int_{(R_{1,\gz}\times
R_{2,\,\gz})^\complement}|\{T_{k_1,\,k_2}(a_R)(x)
\}_{k_1,\,k_2\in\zz}|^q_\ch w(x)\,dx\le C \max \{ b_1^{-\gz\ez},
b_2^{-\gz\ez}\},
\end{eqnarray*}
where $R_{i,\gz}\equiv
x_{R_i}+B^{(i)}_{v_i(\ell(R_i)-1)+u_i+5\sz_i+\gz},\,i=1,\,2.$ Then
$\{T_{k_1,\,k_2}\}_{k_1,\,k_2\in\zz}$ uniquely extends to a bounded
linear operator from $\hp$ to $L^q_{w,\,\ch}(\rnm)$.
\end{lem}

We also need the boundedness result for the anisotropic
Littlewood-Paley $g$-function whose proof is similar
to that of the anisotropic
Lusin-area function; see \cite[Theorem 3.2]{blyz1}.
We omit the details.

\begin{lem}\label{l3.5}
Let $\vz\in\cs_0(\rn)$, $p\in(1,\,\fz)$, and $w\in \ca_p(\rn;\,A)$.
Then, the Littlewood-Paley $g$-function, which is given by $g_\vz(f)(x)
\equiv\lf\{\sum_{k\in\zz}|f\ast\vz_k(x)|^2\r\}^{1/2}$, is bounded on $L^p_w(\rn)$.
\end{lem}

Finally, the isotropic and unweighted versions
of the following lemma have appeared in several
product settings; see, for example, \cite[Theorem 4.3]{hl08}
and the proof of \cite[Proposition 4]{hlll}. In particular, Lemma \ref{l3.7} can
be deduced from the proof of
\cite[Theorem 5.2]{blyz1} as indicated below.

\begin{lem}\label{l3.7}
Suppose that $w\in\wfz$ and $p\in(0,\,1]$. If $f\in L^2(\rnm)\cap\hp,$ then
$f\in L^p_w(\rnm)$. Moreover, there exists a positive constant $C_p$,
independent of $f$, such that $\|f\|_{L^p_w(\rnm)}\le C_p
\|f\|_\hp$.
\end{lem}

\begin{proof}
Let $w\in\wfz$, $p\in(0,\,1]$ and $f\in H^p_w(\rnm)\cap L^2(\rnm)$.
By an argument similar to the proof of \cite[Theorem 4.3]{hl08} or
\cite[Proposition 4]{hlll},
we shall prove that the atomic decomposition of $f$ converges
in $L^2(\rnm)$ and thus pointwise almost everywhere.

Indeed, let $\psi\equiv\psi^{(1)} \otimes \psi^{(2)}$ be
as in Proposition \ref{p2.1}. For any $k\in\zz$, let
$\Oz_k\equiv\{x\in\rnm:\, \vec S_\psi(f)(x)>2^k\},$
$\lz_k\equiv 2^k[w(\Oz_k)]^{1/p}$, and
$$a_k\equiv
\lz_k^{-1}\sum_{P\in m(\wz\Oz_k)} \sum_{R\in\hr_k,
\,R^\ast=P}e_R.$$
Here, our notation is the same
as in \cite[Lemma 4.6]{blyz1}. Since $f\in L^2(\rnm)$,
by Lemma 2.15 and (4.8) of \cite{blyz1}, we have that
$f=\sum_{k\in\zz}\lz_k a_k$ holds in $L^2(\rnm)$,
and hence also almost everywhere. From this,
$\supp a_k\subset\wz\Oz_k'''$ with $w(\wz\Oz_k''')
\ls w(\Oz_k)$ (see \cite[(6.5)]{blyz1}),
$q\in[2,\,\fz)\cap(q_w,\,\fz)$, H\"older's inequality,
and the size condition of $a_k$, it follows that
\begin{eqnarray*} &\|f\|_{L^p_w(\rnm)}^p&\ls
\sum_{k\in\zz}\lz_k^p\int_{\wz\Oz_k'''} |a_k(x)|^p\, w(x)dx
\ls\sum_{k\in\zz}\lz_k^p\|a_k\|_{L^q_w(\rnm)}^p [w(\wz\Oz_k''')]^{1-p/q}\\
&&\ls\sum_{k\in\zz}2^{kp}w(\Oz_k)\ls \|\vec S_\psi(f)\|_{\lp}^p\sim\|f\|_\hp,
\end{eqnarray*}
which completes the proof of Lemma \ref{l3.7}.\end{proof}

\begin{proof}[Proof of Theorem \ref{t1.2}]
Let $T$ be a PASIO of order
$(s_1+1,s_2+1)$ with kernel $K$ as in Definition \ref{d1.3}. By the assumption
\eqref{e1.1} which says that
 $s_i>(q_w/p-1)\log_{|\lz_{i,\,1}|}b_i$ for $i=1,\,2$, we can choose
$1<\lz_{i,\,-}<|\lz_{i,\,1}|$ close to $|\lz_{i,\,1}|$,
$r\in(q_w,\,\fz)$ close to $q_w$ such that
\begin{eqnarray}\label{e3.9}
\eta_i\equiv p[s_i\zeta_{i,\,-}+1]-r>0,\ \ i=1,\,2.
\end{eqnarray}
Let $q>\max\{2,\,r\}$. Then, $(p,\,q,\vec s)_w$ is an admissible
triplet, where $\vec s\equiv(s_1-1,\,s_2-1)$.

Let
$\psi\equiv\psi^{(1)} \otimes \psi^{(2)}$ and $\phi\equiv\phi^{(1)}
\otimes \phi^{(2)}$ be
as in Proposition \ref{p2.1}. By part (iv) of this proposition we have
$\psi=\phi\ast\phi$. Hence, by Theorem \ref{t1.1} and Definition
\ref{d2.6}, $T(f)$ is well defined for any $f\in
L^2_w(\rnm)\cap H^p_w(\rnm)$ and
\begin{eqnarray*}
&\|Tf\|_{H^p_w(\rnm)}&=\|\vec
S_{\phi\ast\phi}(Tf)\|_{L^p_w(\rnm)}\\
&&=\| \{ \phi_{k_1,\,k_2}\ast\phi_{k_1,\,k_2}\ast [T(f)]
\}_{k_1,\,k_2\in\zz}\|_{L^p_{w,\ch}(\rnm)}.
\end{eqnarray*}

To obtain the boundedness of $T$ on $H^p_w(\rnm)$, by Lemma
\ref{l3.4} and the density of $L^2_w(\rnm)\cap H^p_w(\rnm)$ in
$H^p_w(\rnm)$ given by \cite[Theorem 5.1(i)]{blyz1}, it suffices to
prove that for all rectangular $(p,\,q,\,\vec s)_w$-atoms $a$
associated to certain $R\in \hr$ and all $\gz\in\zz_+$,
\begin{eqnarray}\label{e3.10}
&&\quad\int_{(R_{1,\,\gz}\times R_{2,\,\gz})^{\complement}}
\big|\big\{\phi_{k_1,\,k_2}\ast\phi_{k_1,\,k_2}\ast[T(a)](x)
\big\}_{k_1,\,k_2\in\zz}\big|_\ch^p w(x)\, dx\ls
\max\{b_1^{-\eta_1\gz},\, b_2^{-\eta_2\gz}\},
\end{eqnarray}
where $\eta_i$ is as in \eqref{e3.9} and $R_{i,\gz}\equiv
x_{R_i}+B^{(i)}_{v_i(\ell(R_i)-1)+u_i+5\sz_i+\gz}$ for $i=1,\,2$.

The left hand side of \eqref{e3.10} is less than
\begin{eqnarray}\label{e3.11}
&&\Bigg\{\int_{R_{1,\,\gz}^{\complement}\times
R_{2,\,0}^{\complement}}+\int_{R_{1,\,\gz}^{\complement}\times
R_{2,\,0}} +\int_{R_{1,\,0}\times R_{2,\,\gz}^{\complement}}
+\int_{R_{1,\,0}^{\complement}\times R_{2,\,\gz}^{\complement}}
\Bigg\}\\
&&\hs\times
\bigg|\bigg\{\phi_{k_1,\,k_2}\ast\phi_{k_1,\,k_2}\ast[T(a)](x)
\bigg\}_{k_1,\,k_2\in\zz}\bigg|_\ch^p w(x)\,dx\equiv {\mathrm
I}_1+{\mathrm I}_2+{\mathrm I}_3+{\mathrm I}_4.\nonumber
\end{eqnarray}
We only estimate $\mathrm I_2$, since the estimates for the other
three items are similar.

Let $N_1,\,N_2$ be as in Definition \ref{d1.3}. Since
$\phi^{(i)}\in\cs_0(\rr^{n_i})$ for $i=1,\,2$, then by Lemma
\ref{l3.1}, we obtain that $\phi^{(i)}=\sum_{j_i=0}^\fz
b_i^{-3j_i}\phi^{(i,\,j_i)},$ where
$\phi^{(i,\,j_i)}\in\cs_0(\rr^{n_i})$ is a constant multiple of an
$(s_i+N_i+1)$-normalized bump function associated to
$B^{(i)}_{j_i}$. For $j_1,\, j_2\in\zz_+$, let
$\phi^{\{j_1,\,j_2\}}\equiv\phi^{(1,\,j_1)}\otimes \phi^{(2,\,j_2)}$. Thus,
by $\phi=\phi^{(1)}\otimes \phi^{(2)}$, we have
\begin{equation}\label{e3.12}
\phi\ast\phi=\sum_{j_1,\,j_2,\,\ell_1\in\zz_+} b_1^{-3(j_1+\ell_1)}
b_2^{-3j_2}\phi^{\{j_1,\,j_2\}}\ast(\phi^{(1,\,\ell_1)}\otimes \phi^{(2)}).
\end{equation}
Moreover, by Theorem \ref{t1.1} and a density argument,
we obtain
\begin{equation}\label{e3.13}
\phi^{\{j_1,\,j_2\}}\ast(\phi^{(1,\,\ell_1)}\otimes \phi^{(2)})\ast[T(a)]=
K\ast[(\phi^{(1,\,j_1)}\ast_1
\phi^{(1,\,\ell_1)})\otimes\phi^{(2,\,j_2)}]\ast(a\ast_2\phi^{(2)}),
\end{equation}
where $\ast_i$ denotes the convolution on $\rr^{n_i}$, $i=1,\,2$. In
fact, if $a\in\cd(\rnm)$, then the above equality holds. For the
rectangular $(p,\,q,\,\vec s)$-atom $a$, let
$\{a_k\}_{k\in\nn}\subset\cd(\rnm)$ be a sequence of functions
approximating to $a$ in $L^q_w(\rnm)$. Noticing that $T(a_k) \to Ta$
in $L^q_w(\rnm)$, we have \eqref{e3.13}.

For $k_1,\,k_2\in\zz$ and $j_1,\,j_2,\,\ell_1\in\zz_+$, let
$K_{k_1,\,k_2}^{j_1,\,j_2,\,\ell_1}\equiv
K\ast[(\phi^{(1,\,j_1)}\ast_1
\phi^{(1,\,\ell_1)})_{k_1}\otimes\phi_{k_2}^{(2,\,j_2)}]$. By $w\in
\ca_r(\rnm;\,\vec A)$, \cite[Proposition 2.2(i)]{blyz1} and Lemma \ref{l3.3}(iv),
we have $w(R_{1,\,t_1+\gz+1}\times R_{2,\,0})\ls b_1^{r(\gz+t_1)}
w(R).$ From this, \eqref{e3.12}, \eqref{e3.13}, Minkowski's
inequality and H\"older's inequality, it follows that
\begin{eqnarray}\label{e3.14}
\hs\hs\quad {\mathrm I}_2&&\ls
\sum_{j_1,\,j_2,\,\ell_1\in\zz_+}b_1^{-3p(j_1+\ell_1)}b_2^{-3p\ell_2}
\sum_{t_1\in\zz_+}\bigg\{\int_{\lf(R_{1,\,\gz+t_1+1}\setminus
R_{1,\,\gz+t_1}\r)\times R_{2,\,0}}\\
&&\hs\times\bigg|\bigg\{
K_{k_1,\,k_2}^{j_1,\,j_2,\,\ell_1}\ast\lf(a\ast_2
\phi^{(2)}_{k_2}\r)(x)\bigg\}_{k_1,\,k_2\in\zz}\bigg|_\ch^r
w(x)\,dx\bigg\}^{p/r} b_1^{(r-p)(\gz+t_1)}[w(R)]^{1-p/r}\nonumber.
\end{eqnarray}

Let $\wz\ell_1\equiv v_1[\ell(R_1)-1]+u_1+\gz+t_1+5\sz_1$.
Write
\begin{eqnarray*}
&&\bigg|\bigg\{ K_{k_1,\,k_2}^{j_1,\,j_2,\,\ell_1}\ast\lf(a\ast_2
\phi^{(2)}_{k_2}\r)(x)\bigg\}_{k_1,\,k_2\in\zz}\bigg|_\ch^2\\
&&\quad=
\Bigg[\sum_{\gfz{k_1<\wz\ell_1-j_1-\ell_1-4\sz_1}{k_2\in\zz}}
+\sum_{\gfz{k_1\ge\wz\ell_1-j_1-\ell_1-4\sz_1}{k_2\in\zz}}\Bigg]\\
&&\quad\quad\times b_1^{-k_1}b_2^{-k_2}
\int_{B^{(1)}_{k_1}}\int_{B^{(2)}_{k_2}}|K_{k_1,\,k_2}^{j_1,\,j_2,
\,\ell_1}\ast(a\ast_2
\phi^{(2)}_{k_2})(x-y)|^2\,dy\equiv [{\mathrm V}_1(x)]^2+[{\mathrm
V}_2(x)]^2.
\end{eqnarray*}
We only estimate ${\mathrm V}_1$, since the estimate for ${\mathrm
V}_2$ is similar.

For $x\in (R_{1,\,\gz+t_1+1}\setminus R_{1,\,\gz+t_1})\times
R_{2,\,0},\, y\in B^{(1)}_{k_1}\times B^{(2)}_{k_2}$ and $z\in\rnm$,
let $$\wz K_{k_1,\,k_2}^{j_1,\,j_2,\,\ell_1}(z_1,\,z_2)\equiv
K_{k_1,\,k_2}^{j_1,\,j_2,\,\ell_1}(x_1-y_1-A_1^{\wz\ell_1}z_1,\,z_2).$$
For any $\wz y,\,\breve y\in\rnm$, by Taylor's formula with integral
remainder, we have
\begin{eqnarray*}
\wz K_{k_1,\,k_2}^{j_1,\,j_2,\,\ell_1}(\wz y_1,\,\wz y_2)
&&=\sum_{j_1=0}^{s_1-1}\sum_{|\az_1|=j_1}(\wz y_1-\breve
y_1)^{\az_1}\partial^{\az_1}_1 \wz
K_{k_1,\,k_2}^{j_1,\,j_2,\,\ell_1}(\breve y_1,\,\wz y_2)
+\sum_{|\az_1|=s_1 }\int_0^1 (\wz y_1-\breve y_1)^{\az_1}
\nonumber\\
&&\hs\times\partial^{\az_1}_1 \wz
K_{k_1,\,k_2}^{j_1,\,j_2,\,\ell_1}(\breve y_1+r_1(\wz y_1-\breve
y_1),\,\wz y_2) \frac{(1-r_1)^{s_1-1}}{s_1!}\, dr_1. \nonumber
\end{eqnarray*}
Let $\breve y_1\equiv A_1^{-\wz\ell_1}x_{R_1}$ and $\wz
y_1\equiv A_1^{-\wz\ell_1}z_1$.  By $\supp a\subset R''$ and the vanishing
condition of $a$ up to order $s_1-1$, we then have
\begin{eqnarray}\label{e3.15}
&&K_{k_1,\,k_2}^{j_1,\,j_2,\,\ell_1}\ast(a\ast_2 \phi^{(2)}_{k_2})(x-y)\\
&&\hs=\int_\rnm \wz
K_{k_1,\,k_2}^{j_1,\,j_2,\,\ell_1}(A_1^{-\wz\ell_1}z_1,\,x_2-y_2-z_2)(a\ast_2
\phi^{(2)}_{k_2})(z)\,dz\nonumber\\
&&\hs=\sum_{|\az_1|=s_1}\int_0^1\int_{R_1''\times \rrm}(a\ast_2
\phi^{(2)}_{k_2})(z)
(A_1^{-\wz\ell_1}(z_1-x_{R_1}))^{\az_1}\frac{(1+r_1)^{s_1-1}}{s_1!}
\nonumber\\
&&\hs\hs\times
\partial^{\az_1}_1\wz
K_{k_1,\,k_2}^{j_1,\,j_2,\,\ell_1}(A_1^{-\wz\ell_1}[x_{R_1}+(1-r_1)(x_{R_1}-z_1)],
\,x_2-y_2-z_2)\, dz\, dr_1\nonumber.
\end{eqnarray}

Moreover, for $x_1\in R_{1,\,\gz+t_1+1}\setminus R_{1,\,\gz+t_1}$,
$z_1\in R_1''$, $r_1\in(0,\,1)$ and
$\wz\ell_1>k_1+j_1+\ell_1+4\sz_1$, by \eqref{e2.2}, \eqref{e2.3} and
Lemma \ref{l3.3}(iv), we have $\rho_1(x_1-x_{R_1})=b_1^{\wz\ell_1}$
and $\rho_1(z_1-x_{R_1})\le
b_1^{-2\sz_1-t_1-\gz-1}\rho_1(x_1-x_{R_1}),$
which together with $k_1<\wz\ell_1-j_1-\ell_1-4\sz_1$
further means that
$$\rho_1(x_1-x_{R_1}-(1-r_1)(x_{R_1}-z_1))\sim b_1^{\wz\ell_1},$$
and that for $y_1\in B^{(1)}_{k_1}$,
$$\rho_1(x_1-x_{R_1}-(1-r_1)(x_{R_1}-z_1)-y_1)\le  b_1^{k_1}.$$
From this   and Lemma
\ref{l3.2A}, it follows that
\begin{eqnarray*}
&&\big|\partial^{\az_1}_1\wz
K_{k_1,\,k_2}^{j_1,\,j_2,\,\ell_1}(A_1^{-\wz\ell_1}[x_{R_1}+(1-r_1)(x_{R_1}-z_1)],
\,x_2-y_2-z_2)\big|\\
&&\hs=\big|
\partial^{\az_1}_1[
K_{k_1,\,k_2}^{j_1,\,j_2,\,\ell_1}(A_1^{\wz\ell_1}\cdot,\,
x_2-y_2-z_2)](A_1^{-\wz\ell_1}
[x_1-x_{R_1}-(1-r_1)(x_{R_1}-z_1)-y_1])\big|\\
&&\hs\ls b_1^{
(j_1+\ell_1)(1+\ez_1)+k_1\ez_1-\wz\ell_1(1+\ez_1)}
\frac{b_2^{k_2\ez_2}}{\big[b_2^{k_2}
+b_2^{-j_2}\rho_2(x_2-z_2)\big]^{1+\ez_2}},
\end{eqnarray*}
where we used the fact that
$b_2^{k_2}+b_2^{-j_2}\rho_2(x_2-y_2-z_2)\sim
b_2^{k_2}+b_2^{-j_2}\rho_2(x_2-z_2)$.

Furthermore, for $z_1\in R_1''$, by \cite[(2.6)]{blyz1}, we have
$|A_1^{-\wz\ell_1}(z_1-x_{R_1})|\ls b_1^{-(\gz+t_1)\zeta_{1,\,-}}$.
Thus, for $x\in (R_{1,\,\gz+t_1+1}\setminus R_{1,\,\gz+t_1})\times
R_{2,\,0}$, by the above two estimates,
\eqref{e3.15}, $\wz\ell_1=v_1[\ell(R_1)-1]+u_1+\gz+t_1+5\sz_1$, a
similar proof to that of \eqref{e3.8}, and Minkowski's inequality,
we obtain
\begin{eqnarray*}
&{\mathrm V}_1(x)&\ls
\Bigg\{b_1^{2(j_1+\ell_1)(1+\ez_1)-
2\wz\ell_1(1+\ez_1)}\sum_{\gfz{k_1<\wz\ell_1-j_1-\ell_1-4\sz_1}
{k_2\in\zz}}b_1^{2k_1\ez_1}b_2^{j_2(1+\ez_2)} \\
&&\hs\times\bigg[\int_{R_1''}
b_1^{-(\gz+t_1)s_1\zeta_{1,\,-}}\cm^{(2)}(a(z_1,\,\cdot)
\ast_2\phi^{(2)}_{k_2})(x_2)\, dz_1\bigg]^2\Bigg\}^{1/2}\\
&&\ls b_1^{(j_1+\ell_1)
-(t_1+\gz)s_1\zeta_{1,\,-}}b_2^{j_2(1+\ez_2)}\\
&&\hs\times\frac1{b_2^{[v_1\ell(R_1)+t_1+\gz]}}
\int_{R_{1,\,\gz+t_1+1}}\bigg\{\sum_{k_2\in\zz}|\cm^{(2)}(a(z_1,\,\cdot)\ast_2
\phi^{(2)}_{k_2})(x_2)|^2\bigg\}^{1/2}\,dz_1\\
&&\ls b_1^{j_1+\ell_1 -(t_1+\gz) s_1\zeta_{1,\,-}}
b_2^{j_2(1+\ez_2)}\\
&&\hs\times\cm^{(1)}\lf(\bigg\{\sum_{k_2\in\zz}\Big[\cm^{(2)}
\lf(a\ast_2\phi^{(2)}_{k_2}\r)(x_2) \Big]^2\bigg\}^{1/2}\r)(x_1),
\end{eqnarray*}
where and in what follows, $\cm^{(i)}$ denotes the Hardy-Littlewood
maximal function on $\rr^{n_i}$, $i=1,\,2$.

Then, by the above estimate of ${\mathrm V}_1(x)$, the
$L^r_{w(\cdot,\,x_2)}(\rn)$-boundedness of $\cm^{(1)}$ for all
$x_2\in\rrm$, the weighted vector-valued inequality for the
Hardy-Littlewood maximal operator $\cm^{(2)}$ with
$w(x_1,\,\cdot)\in\ca_r(\rrm;\, A_2)$ for all $x_1\in\rn$ (see
\cite[Theorem 2.5]{bh}), Lemma \ref{l3.5} with $g_{\phi^{(2)}}$,
$\supp a\subset R''$, $r>q>1$, H\"older's inequality, and the size
condition of $a$, we have
\begin{eqnarray*}
&&\bigg\{\int_{\lf(R_{1,\,\gz+t_1+1}\setminus
R_{1,\,\gz+t_1}\r)\times R_{2,\,0}}[{\mathrm V}_1(x)]^r\,
w(x)\,dx\bigg\}^{1/r}\\
&&\hs\ls b_1^{j_1+\ell_1 -
(t_1+\gz)s_1 \zeta_{1,\,-}}b_2^{j_2(1+\ez_2)}
\lf\|g_{\phi^{(2)}}(a) \r\|_{L^r_w(\rnm)}\\
&&\hs\ls b_1^{j_1+\ell_1 - (t_1+\gz) s_1\zeta_{1,\,-}}
b_2^{j_2(1+\ez_2)}
\|a\|_{L^q_w(\rnm)}[w(R'')]^{1/r-1/q}\\
&&\hs\ls b_1^{j_1+\ell_1  - (t_1+\gz)s_1 \zeta_{1,\,-}}
b_2^{j_2(1+\ez_2)} [w(R)]^{1/r-1/p}.
\end{eqnarray*}

From this and $\eta_1=p[(s_1+1)\zeta_{1,\,-}+1]-r>0$, it
follows that
\begin{eqnarray*}
&&\sum_{t_1\in\zz_+}\bigg\{\int_{\lf(R_{1,\,\gz+t_1+1}\setminus
R_{1,\,\gz+t_1}\r)\times R_{2,\,0}}[{\mathrm V}_1(x)]^r
w(x)\,dx\bigg\}^{p/r}
b_1^{(r-p)(\gz+t_1)}[w(R)]^{1-p/r}\\
&&\hs\ls\sum_{t_1\in\zz_+}b_1^{p(j_1+\ell_1)
}b_2^{pj_2(1+\ez_2)}
[w(R)]^{p/r-1}b_1^{-p  (t_1+\gz)s_1\zeta_{1,\,-}}b_1^{(t_1+\gz)(r-p)}
[w(R)]^{1-p/r}\\
&&\hs\ls b_1^{p(j_1+\ell_1) -\gz\eta_1} b_2^{pj_2(1+\ez_2)},
\end{eqnarray*}
which together with \eqref{e3.14}  yields that
${\mathrm I}_2\ls b_1^{-\gz\eta_1}$.

By an estimate similar to that of ${\mathrm I}_2$, we also have
${\mathrm I}_1+{\mathrm I}_3+{\mathrm
I}_4\ls\max\{b_1^{-\gz\eta_1},\,b_2^{-\gz\eta_2}\}$, where $\eta_1$
and $\eta_2$ are as \eqref{e3.9}. Thus, by this and \eqref{e3.11},
we obtain \eqref{e3.10} and hence the boundedness of $T$ on
$H^p_w(\rnm;\,\vec A)$.

Finally, let us prove that $T$ is bounded from $\hp$ to
$L_w^p(\rnm)$ with $p\in(0,\,1]$ satisfying \eqref{e1.1}
by borrowing some ideas from the proof of \cite[Theorem 1.11]{hl08}.
Assume that $f\in L^2(\rnm)\cap\hp$. By Theorem \ref{t1.1},
Lemma \ref{l3.7} and the boundedness of $T$ on $\hp$,  we obtain
that
$$Tf\in L^p_w(\rnm)\cap\hp\cap L^2(\rnm)$$
and $\|Tf\|_\lp\ls \|Tf\|_\hp\ls \|f\|_\hp.$ This together with the
density of $L^2(\rnm)\cap\hp$ in $\hp$ given by \cite[Theorem 5.1
(i)]{blyz1} implies that $T$ extends to a linear bounded operator
from $H^p_w(\rnm )$ to $L^p_w(\rnm)$.
This finishes the proof of
Theorem \ref{t1.2}.
\end{proof}

\bigskip

\noindent Baode Li

\medskip

\noindent School of Mathematics and System
Sciences, Xinjiang University, Urumqi {\rm 830046},
China

\smallskip

\noindent {\it E-mail address}: \texttt{baodeli1981@sina.com}

\bigskip

\noindent Marcin Bownik

\medskip

\noindent Department of Mathematics, University of Oregon, Eugene,
OR 97403-1222, USA

\smallskip

\noindent {\it E-mail address}: \texttt{mbownik@uoregon.edu}

\bigskip

\noindent Dachun Yang (Corresponding author) and Yuan Zhou

\medskip

\noindent School of Mathematical Sciences, Beijing Normal
University, Laboratory of Mathematics and Complex Systems, Ministry
of Education, Beijing 100875, People's Republic of China

\smallskip

\noindent{\it E-mail addresses}: \texttt{dcyang@bnu.edu.cn, yuanzhou@mail.bnu.edu.cn}

\end{document}